\documentclass[11pt]{article}
\usepackage{amsmath}
\usepackage{amsthm}
\usepackage{amssymb,amsfonts}
\usepackage{hyperref}
\usepackage{latexsym}
\usepackage{todonotes}
\usepackage{nicefrac}
\usepackage{epsfig}
\usepackage{stmaryrd}
\usepackage{setspace}
\usepackage{enumerate}
\usepackage[all]{xypic}
\usepackage{bbm,ifpdf,tikz}
\ifpdf
\usepackage{pdfsync}
\fi

\newtheorem{theorem}{Theorem}[section]

\newtheorem{lemma}[theorem]{Lemma}
\newtheorem{proposition}[theorem]{Proposition}
\newtheorem{corollary}[theorem]{Corollary}

\newtheorem{conjecture}[theorem]{Conjecture}

{
\theoremstyle{definition}

\newtheorem{example}[theorem]{Example}

\newtheorem{remark}[theorem]{Remark}
}

\newcommand{\excise}[1]{}

\newcommand{\id}{\operatorname{id}}

\renewcommand{\dim}{\operatorname{dim}}

\newcommand{\crk}{\operatorname{crk}}

\renewcommand{\and}{\qquad\text{and}\qquad}

\newcommand{\Hom}{\operatorname{Hom}}

\newcommand{\CG}{\Comp_{p,i}^{\Gamma\!,\,\Delta}}

\newcommand{\R}{\mathbb{R}}
\newcommand{\C}{\mathbb{C}}

\newcommand{\IH}{I\! H}

\newcommand{\cI}{\mathcal{I}}

\newcommand{\la}{\lambda}

\newcommand{\FS}{\operatorname{FS}}
\newcommand{\FA}{\operatorname{FA}}
\newcommand{\FSop}{\operatorname{FS^{op}}}
\newcommand{\FAop}{\operatorname{FA^{\!op}}}
\newcommand{\Conf}{\operatorname{Conf}}

\newcommand{\Aut}{\operatorname{Aut}}

\newcommand{\Comp}{\operatorname{Comp}}

\begin{document}
\spacing{1.2}
\noindent{\Large\bf Configuration spaces, \boldmath{$\FSop$}-modules, and 
Kazhdan-Lusztig\\ polynomials of braid matroids
}\\

\noindent{\bf Nicholas Proudfoot 
and Benjamin Young}\\
Department of Mathematics, University of Oregon,
Eugene, OR 97403\\

{\small
\begin{quote}
\noindent {\em Abstract.}
The equivariant Kazhdan-Lusztig polynomial of a braid matroid may be interpreted as the intersection
cohomology of a certain partial compactification of the configuration space of $n$ distinct labeled points in $\C$,
regarded as a graded representation of the symmetric group $S_n$.  We show that, in fixed cohomological degree,
this sequence of representations of symmetric groups naturally admits the structure of an $\FS$-module,
and that the dual $\FSop$-module is finitely generated.  Using the work of Sam and Snowden,
we give an asymptotic formula for the dimensions of these representations and obtain restrictions on which
irreducible representations can appear in their decomposition.
\end{quote} }

\section{Introduction}
Given a matroid $M$, the Kazhdan-Lusztig polynomial $P_M(t)$ was defined in \cite{EPW}.
More generally, if $M$ is equipped with an action of a finite group $W$,
one can define the $W$-equivariant Kazhdan-Lusztig polynomial $P_M^W(t)$ \cite{GPY}.
By definition, $P_M^W(t)$ is a graded virtual representation of $W$, and taking dimension recovers
the non-equivariant polynomial.  These representations have been computed when $M$ is a uniform matroid
\cite[Theorem 3.1]{GPY} and conjecturally for certain graphical matroids
\cite[Conjecture 4.1]{thag}.
However, in the case of the braid matroid (the matroid associated with the complete graph on $n$ vertices), 
very little is known.  The non-equivariant version of this problem was taken up in \cite[Section 2.5]{EPW} and the $S_n$-equivariant
version in \cite[Section 4]{GPY}, but with few concrete results or even conjectures.

In this paper we use an interpretation of the equivariant Kazhdan-Lusztig polynomial of the braid matroid $M_n$ as the intersection cohomology of a certain partially compactified configuration space to show that, in fixed cohomological degree, it admits the structure of an
$\FS$-module, as studied in \cite{Pirashvili, church-ellenberg-farb, sam}.  Applying the results of Sam and Snowden \cite{sam},
we use the $\FS$-module structure (or, more precisely, the dual $\FSop$-module structure) to improve our understanding of this sequence
of representations.  In particular, we obtain the following results (Corollary \ref{cor}):
\begin{itemize}
\item For fixed $i$, we prove that the generating function for the $i^\text{th}$ non-equivariant Kazhdan-Lusztig coefficient of $M_n$
(with $n$ varying) is a rational function with poles lying in a prescribed set.
\item For fixed $i$, we derive an asymptotic formula for the $i^\text{th}$ non-equivariant Kazhdan-Lusztig coefficient of $M_n$
in terms of another Kazhdan-Lusztig coefficient that depends only on $i$.
\item We show that, if $\la$ is a partition of $n$ and the associated Specht module $V_\la$ appears as a summand
of the $i^\text{th}$ equivariant Kazhdan-Lusztig coefficient
of $M_n$, then $\la$ has at most $2i$ rows.
\end{itemize} 
We also produce relative versions of these results in which we start with an arbitrary graph $\Gamma$ and consider the sequence
of graphs whose $n^\text{th}$ element is obtained from $\Gamma$ by adding $n$ new vertices and connecting them to everything
(including each other).  The original problem is the special case where $\Gamma$ is the empty graph.

\vspace{\baselineskip}
\noindent
{\em Acknowledgments:}
The authors are grateful to Steven Sam and John Wiltshire-Gordon for extremely helpful discussions
without which this paper would not have been written, and to Tom Braden
for greatly clarifying the material in Section \ref{sec:local}.
The first author is supported by NSF grant DMS-1565036.  

\section{Kazhdan-Lusztig polynomials and configuration spaces}\label{sec:kl}
Let $M$ be a matroid on the ground set $\cI$, equipped with an action of a finite group $W$.  This means that $W$ acts on $\cI$ by permutations
and that the action of $W$ takes bases to bases.  An {\bf equivariant realization} of $W\curvearrowright M$ is $W$-subrepresentation $V\subset \C^\cI$
such that $B\subset \cI$ is a basis for $M$ if and only if $V$ projects isomorphically onto $\C^B$.

Note that we have $\C^\cI\subset \left(\C P^1\right)^\cI$, sitting inside as the locus of points with no coordinate
equal to $\infty$.  More generally, for any subset $S\subset\cI$, let
$p_S\in \left(\C P^1\right)^\cI$ be the point with $(p_S)_i=0$ for all $i\in S$ and $(p_S)_j = \infty$ for all $j\in S^c$,
and let
$$\C^\cI_S := \left\{p\in\left(\C P^1\right)^\cI\;\Big{|}\; \text{$p_i\neq\infty$ for all $i\in S$ and $p_i\neq 0$ for all $i\in S^c$}\right\}$$
be the standard affine neiborhood of $p_S$.  Thus $p_\cI = 0\in V\subset\C^\cI = \C^\cI_\cI$.
Given a $W$-subrepresentation $V\subset\C^\cI$, we define
the following three spaces with $W$-actions:
\begin{itemize}
\item $U(V) := V\cap (\C^\times)^\mathcal{I}$, the {\bf complement} of the coordinate hyperplane arrangement in $V$,
\item $Y(V) := \overline V \subset \left(\C P^1\right)^\cI$, the {\bf Schubert variety} of $V$ (see \cite{ArBoo} or \cite[Section 7]{PXY}),
\item $X(V) := Y(V) \cap \C^\cI_{\emptyset}$, the {\bf reciprocal plane} of $V$.
\end{itemize}
Note that $Y(V)$ is a compactification of $U(V)$, while $V$ and $X(V)$ are each partial compactifications of $U(V)$.

Let $C_{M,i}^W$ denote the coefficient of $t^i$ in the equivariant Kazhdan-Lusztig polynomial $P_M^W(t)$ of $W\curvearrowright M$.
The following theorem appears in \cite[Corollary 2.12]{GPY} as an application of the work in \cite[Section 3]{PWY}.

\begin{theorem}\label{ih}
If $V\subset \C^\cI$ is an equivariant realization of $W\curvearrowright M$, then 
$C_{M,i}^W$ is isomorphic as a representation of $W$ to the intersection cohomology group $I\! H^{2i}\big(X(V); \C\big)$.
In particular, $C_{M,i}^W$ is an honest (not just virtual) representation.
\end{theorem}

Let $\cI_n := \big\{(i,j) \mid i\neq j \in [n]\big\}$, and let $M_n$ be the matroid
on the ground set $\cI_n$ whose bases consist of oriented spanning trees for the complete graph on $n$ vertices.  
We will refer to $M_n$ as the {\bf braid matroid}, which comes equipped with a natural
action of the symmetric group $S_n$.

\begin{remark}
It is more standard to define the braid matroid on the ground set of {\em unordered} pairs of elements of $[n]$.
Our matroid $M_n$ is not simple (for any $i\neq j$, the set $\{(i,j), (j,i)\}$ is dependent), and its simplification is 
$S_n$-equivariantly isomorphic to the usual braid matroid.  In particular, they have 
the same lattice of flats (see Section \ref{sec:local} for the definition of a flat), 
and therefore the same equivariant Kazhdan-Lusztig polynomial.
We prefer the ordered version because it is equivariantly realizable (as we explain below), thus we may apply Theorem \ref{ih}.
\end{remark}

Consider the linear map $f:\C^n\to \C^{\cI_n}$ given by $f_{ij}(z_1,\ldots,z_n) = z_i - z_j$.  
The kernel of $f$ is equal to the diagonal line $\C_\Delta\subset \C^n$,
so $f$ descends to an inclusion of $V_n := \C^n/\C_{\Delta}$ into $\C^{\cI_n}$, which gives an equivariant realization of $\C^n$.
Let $U_n := U(V_n)$, $Y_n := Y(V_n)$, and $X_n := X(V_n)$.
The space $U_n$ may be identified with the configuration space of $n$ distinct labeled points in $\C$, modulo simultaneous translation.
Informally, $V_n$ is obtained from $U_n$ by allowing the distances between points
to go to zero, the reciprocal plane $X_n$ is obtained from $U_n$ by allowing the distances between points to go to infinity,
and the Schubert variety $Y_n$ is obtained from $U_n$ by allowing distances between points to go to either zero or infinity.

\begin{remark}
The reciprocal plane $X_n$ may also be described as the spectrum 
of the subring $\C\left[\frac{1}{x_i-x_j}\mid i\neq j\right]$ of the ring of rational functions
on $\C^n$.  More generally, $X(V)$ is isomorphic to the spectrum of the subring of rational functions on $V$
generated by the reciprocals of the coordinate functions.
This ring is called the {\bf Orlik-Terao algebra} of $V\subset \C^\cI$.
\end{remark}

The non-equivariant Kazhdan-Lusztig polynomial of $M_n$ for $n\leq 20$ appears in \cite[Section A.2]{EPW}.
The first few coefficients of this polynomial can be expressed in terms of Stirling numbers \cite[Corollary 2.24 and Proposition 2.26]{EPW}.
The same can be said of all of the terms, but the expressions become increasingly complicated.
Indeed, the $i^\text{th}$ coefficient can be expressed as an alternating
sum of $i$-fold products of Stirling numbers, where the number of summands is equal to $2\cdot 3^{i-1}$ \cite[Corollary 4.5]{PXY}.
We also made a conjecture about the leading term when $n$ is even \cite[Section A]{EPW}.  The degree of the Kazhdan-Lusztig polynomial
is by definition strictly less than half of the rank of the matroid, so the largest possible degree of $P_{M_{2i}}(t)$ is $i-1$.

\begin{conjecture}\label{top}
For all $i>0$, $C_{M_{2i}, i-1} = (2i-3)!!(2i-1)^{i-2}$,
the number of labeled triangular cacti on $(2i-1)$ nodes \cite[Sequence A034941]{oeis}.
\end{conjecture}

The equivariant Kazhdan-Lusztig polynomial of the braid matroid is even more difficult to understand.
The linear term is computed in \cite[Proposition 4.4]{GPY}, and we also compute the remaining coefficients for $n\leq 9$
\cite[Section 4.3]{GPY}.  We also give a functional equation that characterizes the generating function for the Frobenius
characteristics of the equivariant Kazhdan-Lusztig polynomials \cite[Equation (7)]{GPY}, but we do not know how to solve this equation.

\section{The spectral sequence}\label{sec:local}
In this section we explain how to construct a spectral sequence to compute the intersection cohomology of the reciprocal plane,
which we will later use to endow the Kazhdan-Lusztig coefficients of braid matroids with an FS-module structure.
This construction appears for a particular example in \cite[Section 3]{PWY}, and we make some remarks there about
how to generalize the construction to arbitrary $V\subset\C^\cI$.  We will give the construction in full generality here,
taking care to emphasize the functoriality, which will be crucial for our application in Section \ref{sec:kl}. 

A subset $F\subset\cI$ is called a {\bf flat} of $M$ if there exists a point $v\in V$ such that $F = \{i\mid v_i=0\}$.
Given a flat $F$, let $V^F := V\cap \C^{F^c} \subset \C^{F^c}$ and let $V_F \subset \C^F$ be the image of $V$ along the projection
$\C^\cI\twoheadrightarrow\C^F$.  The dimension of $V_F$ is called the {\bf rank} of $F$, while the dimension of $V^F$ is called the {\bf corank}. 

Given a flat $F\subset\cI$, let $Y(V)_F := \big\{p\in Y(V) \mid p_i = \infty \Leftrightarrow i\in F^c\big\}$.
Then we have 
\begin{equation}\label{Y-strat}
Y(V) = \bigsqcup_F Y(V)_F\end{equation} 
and $Y(V)_F \cong V_F$ for all $F$ \cite[Lemmas 7.5 and 7.6]{PXY}.
This affine paving may also be described as the orbits of a group action.
The additive group $\C$ acts on $\C P^1 = \C\cup\{\infty\}$ by translations; taking products, we obtain an action
of $\C^\cI$ on $\left(\C P^1\right)^\cI$.  The subgroup $V\subset\C^\cI$ acts on the subvariety 
$Y(V) := \overline V \subset \left(\C P^1\right)^\cI$, and the subset $Y(V)_F$ is equal to the orbit of the point $p_F\in Y(V)$.
The stabilizer of $p_F$ is equal to $V^F\subset V$, and the orbit is therefore isomorphic to $V/V^F\cong V_F$.

For any flat $F\subset\cI$, there is a canonical inclusion $\epsilon_F: X(V^F)\hookrightarrow Y(V)\cap \C^\cI_F$
defined explicitly by the formula
$$\epsilon_F(p) := \begin{cases}
p_i\;\;\text{if $i\in F^c$}\\
0\;\;\;\text{if $i\in F$}.
\end{cases}$$
In particular, $\epsilon_F(\infty) = p_F$.  Consider the map
$$\begin{aligned}
\varphi_F:V \times X(V^F) &\longrightarrow Y(V)\\
(v, p) &\longmapsto v\cdot \epsilon_F(p).
\end{aligned}
$$
If we choose a section $s:V_F\to V$ of the projection $\pi_F:V\to V_F$,
then the restriction of $\varphi_F$ to $s(v_F)\times X(V^F)$ is an open immersion.
In particular, for every $v\in V$, the map 
$\varphi_{F,v}:X(V^F) \to Y(V)$ taking $p$ to $\varphi_F(v,p)$
is a normal slice to the stratum $V_F\subset Y(V)$ at the point $\varphi_{F,v}(\infty) = \pi_F(v)\in V_F$.

Intersecting the stratification in Equation \eqref{Y-strat} with $\C^\cI_\emptyset$, we obtain a stratification
$$X(V) = \bigsqcup_F U(V_F)$$ of the reciprocal plane $X(V)$, which can be used
to construct a spectral sequence that computes the intersection cohomology of $X(V)$.

\begin{theorem}\label{ss}
Let $W$ be a finite group acting on $\cI$, and let $V\subset \C^\cI$ be a $W$-subrepresentation.
There exists a first quadrant cohomological spectral sequence $E(V,i)$ in the category of $W$-representations 
with
$$E(V,i)_1^{p,q} = \bigoplus_{\crk F = p} H^{2i-p-q}\big(U(V_F); \C\big) \otimes I\! H^{2(i-q)}\big(X(V^F); \C\big),
$$
converging to $\IH^{2i}(X(V); \C)$.
\end{theorem}

\begin{proof}
Let $\iota_F:V_F\to Y(V)$ denote the inclusion of the stratum of $Y(V)$ indexed by $F$,
which restricts to the inclusion $\iota_F:U(V_F)\to X(V)$ of the corresponding stratum of $X(V)$.
The stratification of $X(V)$ induces a filtration by supports on the complex of global sections of an injective resolution
of the intersection cohomology sheaf $IC_{X(V)}$.  This filtered complex gives rise to a spectral sequence $E(V)$ with
$$E(V)_1^{p,q} = \bigoplus_{\crk F = p}\mathbb{H}^{p+q}\left(\iota_F^! IC_{X(V)}\right)$$
converging to $I\! H^*(X(V); \C)$ \cite[Section 3.4]{BGS96}.

The sheaf $\iota_F^! IC_{X(V)}$ is {\em a priori} a local system on $U(V^F)$ with fibers equal to the compactly supported
intersection cohomology of the stalks of $IC_{X(V)}$.
However, since $X(V)$ is open in $Y(V)$, the sheaf $\iota_F^! IC_{X(V)}$ on $U(V_F)$ coincides
with the restriction of the sheaf $\iota_F^! IC_{Y(V)}$ on $V_F$.
Since $V_F$ is a vector space, this local system is trivial.  Even better, we have a canonical trivialization.
For any $v_F\in V_F$, we can choose $v\in V$ with $\pi_F(v) = v_F$, and the slice
$\varphi_{F,v}:X(V^F)\to Y(V)$ induces an isomorphism from the fiber of $\iota_F^! IC_{Y(V)}$ to 
the compactly supported intersection cohomology group $I\! H^*_c\big(X(V^F); \C\big)$.
Since the kernel $V^F$ of $\pi_F$ is connected, this isomorphism does not depend on the choice of $v$.
Thus we have a canonical isomorphism
$$E(V)_1^{p,q} = \bigoplus_{\crk F = p}\bigoplus_{j+k=p+q} H^j\big(U(V_F); \C\big) \otimes I\! H^k_c\big(X(V^F); \C\big).$$

We now consider the weight filtration on $E(V)$, and pass to the maximal subquotient $E(V,i)$ of weight $2i$.
The group $H^j\big(U(V_F); \C\big)$ is pure of weight $2j$ \cite{Shapiro};
the groups $I\! H^k_c\big(X(V^F); \C\big)$ and $I\! H^k\big(X(V); \C\big)$
are both pure of weight $k$, and they vanish when $k$ is odd \cite[Proposition 3.9]{EPW}.
This implies that 
$$E(V,i)_1^{p,q} = \bigoplus_{\crk F = p} H^{2j-p-q}\big(U(V_F); \C\big) \otimes I\! H^{2(p+q-i)}_c\big(X(V^F); \C\big),$$
and that $E(V,i)$ converges to $I\! H^{2i}\big(X(V); \C\big)$.  Finally, we observe that $\dim X(V^F) = \crk F = p$, so
Poincar\'e duality tells us that
$I\! H^{2(p+q-i)}_c\big(X(V^F); \C\big) \cong I\! H^{2(i-q)}\big(X(V^F); \C\big)$.
\end{proof}

\begin{remark}
The proof of Theorem \ref{ss} for a particular class of examples appears in \cite[Proposition 3.3]{PWY}.
The argument here is essentially the same.  Indeed, we implicitly used Theorem \ref{ss} in the proof of Theorem \ref{ih},
which originally appeared in \cite[Corollary 2.12]{GPY}.  The only new ingredient here is an emphasis of the fact that the
local system $\iota_F^! IC_{X(V)}$ is {\em canonically} trivialized, which we need in order to make
sense of Theorem \ref{canon}.  We are grateful to Tom Braden for explaining to us how this works.
\end{remark}

Next, we will show that for every flat $F\subset\cI$, we obtain a canonical map from $E(V,i)$
to $E(V^F,i)$, which we will describe explicitly.
The cohomology of $U(V)$ is generated by degree 1 classes $\{\omega_i \mid i\in\cI\}$. 
Explicitly, we have $\omega_i = [d\log z_i]$, where $z_i$ is the coordinate function on $U(V)\subset\C^\cI_{\cI}$.

\begin{theorem}\label{canon}
Suppose that $F\subset\cI$ is a flat.
\begin{enumerate}
\item There is a canonical map of graded vector spaces $I\! H^*\big(X(V); \C\big)\to I\! H^*\big(X(V^F); \C\big)$,
equivariant for the stabilizer $W_F\subset W$ of $F$.
\item There is a canonical map of spectral sequences $E(V,i)\to E(V^F,i)$, equivariant for
the stabilizer $W_F\subset W$ of $F$, converging to the map in part 1.
\item If $G\supset F$, then the compositions $I\! H^*\big(X(V); \C\big)\to I\! H^*\big(X(V^F); \C\big)\to I\! H^*\big(X(V^G); \C\big)$
and $E(V,i)\to E(V^F,i)\to E(V^G,i)$ coincide with the maps
$I\! H^*\big(X(V); \C\big)\to I\! H^*\big(X(V^G); \C\big)$ and $E(V,i)\to E(V^G,i)$, respectively.
\item The map from
$$E(V,i)_1^{p,q} = \bigoplus_{\crk G = p} H^{2i-p-q}\big(U(V_G); \C\big) 
\otimes I\! H^{2(i-q)}\big(X(V^G); \C\big)$$
to
$$E(V^F,i)_1^{p,q} = \bigoplus_{\substack{G\supset F\\\crk G = p}} H^{2i-p-q}\big(U(V^F_G); \C\big) 
\otimes I\! H^{2(i-q)}\big(X(V^G); \C\big)$$
kills summands with $G\not\supset F$.  If $G\supset F$ and $i\in G$, then the map on $G$ summands
is induced by the map $H^{1}\big(U(V_G); \C\big)\to H^{1}\big(U(V^F_G); \C\big)$ obtained by setting $\omega_i$
equal to zero for all $i\in F$.
\end{enumerate}
\end{theorem}

\begin{proof}
For any point $v_F\in U(V_F)\subset V_F$, we have a map
$$I\! H^*\big(X(V); \C\big) \to H^*\big(IC_{X(V),v_F}\big) \cong H^*\big(IC_{Y(V),v_F}\big) 
\cong H^*\big(IC_{X(V^F),\infty}\big)
\cong I\! H^*\big(X(V^F); \C\big),$$
where the second isomorphism is induced by the slice 
$\varphi_{F,v}:X(V^F)\to Y(V)$ for any $v\in V$ such that $\pi_F(v) = v_F$
and the third isomorphism is induced by the contracting action of $\C^\times$ on $X(V^F)$ \cite[Corollary 1]{Springer-purity}.
As before, the fact that this map is independent of the
choice of $v$ follows from the fact that the kernel $V^F$ of $\pi_F$ is connected.  Since the codimension $p$ strata
of $X(V^F)$ coincide with the preimages of the codimension $p$ strata of $Y(V)$, 
the filtrations of $IC_{Y(V),v_F}
\cong IC_{X(V^F),\infty}$ induced by the two stratifications coincide, thus
this map induces a map of spectral sequences
associated with the stratifications.  This proves the first two parts of the theorem.

To prove the third part of the theorem, choose generic elements $v,v'\in V$ and $v''\in V^F$ such that $v = v' + v''$.
We then have maps $$\varphi_{G,v}:X(V^G)\to Y(V),\qquad \varphi_{F,v'}:X(V^F)\to Y(V),\and
\varphi_{G,v''}^F:X(V^G)\to Y(V^F).$$
If $p\in X(V^G)$ is sufficiently close to the point $\infty$ (more precisely, if $|p_i| > |v''_i|$ for all $i\in G^c$),
then $\varphi_{G,v''}^F(p)\in X(V^F)$.  Thus the composition
$\varphi_{F,v'}\circ\varphi_{G,v''}^F$ is well defined in a neighborhood of $\infty\in X(V^G)$, and on that neighborhood
we have $$\varphi_{G,v} = \varphi_{F,v'}\circ\varphi_{G,v''}^F.$$
Since the maps in parts 1 and 2 are determined by the behavior of the slice map in a neighborhood of $\infty$,
this implies that the maps compose as desired.

To prove the last part of the theorem, we need to understand explicitly the map from the $G$ stratum of $X(V^F)$
to the $G$ stratum of $Y(V)$.  Specifically, if $p\in U(V^F_G)$, and $i\in G$,
then $$\varphi_{F,v}(p)_i = \begin{cases}
p_i + v_i\;\;\;\text{if $i\in F^c$}\\
v_i\;\;\;\;\;\;\;\;\;\;\,\text{if $i\in F$.}
\end{cases}
$$
As in the previous paragraph, if we 
restrict to the open set $B\subset U(V^F_G)$ on which each $p_i$ has norm larger than $|v_i|$, then our map
will take values in $U(V_G)$.  Note that $B$ is homotopy equivalent to $U(V^F_G)$, and the map in the spectral
sequence is determined by the pullback map from $H^*\big(U(V_G); \C)$ to $H^*(B; \C) \cong H^*\big(U(V^F_G);\C \big)$.

Let $z_i$ be the $i^\text{th}$ coordinate function on $U(V_G)$, so that $\omega_i = [d\log z_i]$.
If $i\in F$, then $z_i$ pulls back to a constant function, so $\omega_i$ pulls back to zero.  If $i\in G\smallsetminus F$,
then $z_i$ pulls back to $z_i - v_i$, so $\omega_i$ pulls back to
$$[d\log(z_i - v_i)] = [d\log(z_i \cdot (1-v_i/z_i))] = [d\log z_i] + [d\log(1-v_i/z_i)] = \omega_i + [d\log(1-v_i/z_i)].$$
Since the norm of $z_i$ is always greater than the norm of $v_i$ on $B$, the real part of $1-v_i/z_i$ is always positive,
which implies that $d\log(1-v_i/z_i)$ is exact.  Thus $\omega_i$ pulls back to $\omega_i$, as desired.
\end{proof}

We now unpack Theorem \ref{ss} in the special case where $\cI = \cI_n$ and $V = V_n$.
In this case, flats are in bijection with set-theoretic partitions of $[n]$.  More precisely, given a partition of $[n]$,
the set of all ordered pairs $(i,j)$ such that $i$ and $j$ lie in the same block of the partition is a flat, and every flat arises in this way.
A flat of corank $p$ corresponds to a partition into $p+1$ (unlabeled) blocks $P_1,\ldots,P_{p+1}$.  Given such a flat $F$, we have
$U((V_n)_F) \cong U_{|P_1|}\times\cdots\times U_{|P_{p+1}|}$ and $X(V_n^F) \cong X_{p+1}$.  
In order to clarify the issue of labeled versus unlabeled partitions, we make the following definitions:
$$A_i^{p,q}(n) := \bigoplus_{f:[n]\twoheadrightarrow[p+1]} H_{2i-p-q}\Big(U_{|f^{-1}(1)|}\times\cdots\times U_{|f^{-1}(p+1)|}; \C\Big) \otimes I\! H_{2(i-q)}(X_{p+1}; \C)$$
and $$B_i^{p,q}(n) := A_i^{p,q}(n)^{S_{p+1}},$$
where $S_{p+1}$ acts on $[p+1]$.
Thus we have the following corollary of Theorem \ref{ss}.

\begin{corollary}\label{ss-braid}
There exists a first quadrant cohomological spectral sequence $E(n,i)$ in the category of $S_n$-representations with
$E(n,i)_1^{p,q} = B_i^{p,q}(n)^*$
converging to
$\IH^{2i}(X_n).$
\end{corollary}

\begin{remark}
The reason for using homology rather than cohomology in the definition of $A_i^{p,q}(n)$ (and then undoing this by dualizing in Corollary \ref{ss-braid})
will become clear in Section \ref{sec:kl}.  Briefly, the explanation is that intersection cohomology admits the structure of an FS-module and intersection
homology admits the structure of an $\FSop$-module, and it is the $\FSop$-module structure that will prove to be more useful.
\end{remark}

\section{FS-modules and {\boldmath{$\FSop$}}-modules}
Let FS be the category whose objects are nonempty finite sets and whose morphisms are surjective maps.
An FS-module is a covariant functor from FS to the category of complex vector spaces, and an $\FSop$-module
is a contravariant functor from FS to the category of complex vector spaces.  If $N$ is an FS-module
or an $\FSop$-module, we write $N(n) := N([n])$, which we regard as a representation of the symmetric group
$S_n = \Aut_{\FS}([n])$.  Let FA be the category whose objects are nonempty finite sets and whose morphisms are all maps.

For any positive integer $m$, let $P_m := \C\{\Hom_{\FS}(-,[m])\}$ be the $\FSop$-module that takes a finite set $E$
to the vector space with basis given by surjections from $E$ to $[m]$; this is a projective $\FSop$-module called the {\bf principal
projective} at $m$.  We say that an $\FSop$-module $N$ is {\bf finitely generated} if it is isomorphic to the quotient
of a finite sum of principal projectives, and we say that it is {\bf finitely generated in degrees \boldmath$\leq d$} if one only needs
to use $P_m$ for $m\leq d$.  This is equivalent to the statement that, for any finite set $E$ and any vector $v\in N(E)$, we can write
$v$ as a finite linear combination of elements of the form $f^*(x)$, where $f:E\twoheadrightarrow [m]$ and $x\in N(m)$ for some $m\leq d$.

We call an $\FSop$-module {\bf {\boldmath$d$}-small} if it is a subquotient of a module
that is finitely generated in degrees $\leq d$.  A $d$-small $\FSop$-module is always finitely generated 
\cite[Corollary 8.1.3]{sam}, but not necessarily in degrees $\leq d$.  

For any partition $\la = (\la_1,\ldots,\la_{\ell(\la)})\vdash n$, let $V_\la$ be the corresponding irreducible representation of $S_n$.
If $\la$ is a partition of $k$ and $n\geq k + \la_1$, let $\la(n)$ be the partition of $n$ obtained by adding a part of size $n-k$.
For any $\FSop$-module $N$, consider the ordinary generating function
$$H_N(u) := \sum_{n=1}^\infty u^n \dim N(n),$$
and the exponential generating function
$$G_N(u) := \sum_{n=1}^\infty \frac{u^n}{n!} \dim N(n).$$
For any natural number $d$, let $$r_d(N) := \lim_{n\to\infty} \frac{\dim N(n)}{d^n},$$
which may or may not exist.
The statements and proofs of the following results were communicated to us by Steven Sam.

\begin{theorem}\label{dsmall}
Let $N$ be a $d$-small $\FSop$-module.
\begin{enumerate}
\item If $\la\vdash n$ and $\Hom_{S_n}(V_\la, N(n)) \neq 0$, then $\ell(\la)\leq d$.  
\item For any partition $\la$ with $n \geq |\la| + \la_1 $, 
$\dim\Hom_{S_n}\!\big(V_{\la(n)}, N(n)\big)$ is bounded by a polynomial in $n$ of degree at most $d-1$.
\item The ordinary generating function $H_N(u)$
is a rational function whose poles are contained in the set $\{1/j \mid 1\leq j \leq d\}$.
\item There exists polynomials $p_0(u),\ldots,p_{d}(u)$ such that the exponential generating function
$G_N(u)$ is equal to $$\sum_{j=0}^d p_j(u)e^{ju}.$$
\item The function $H_N(u)$ has at worst a simple pole at $1/d$.  Equivalently, the limit $r_d(N)$
exists, and the polynomial $p_d(u)$
in statement 4 is the constant function with value $r_d(N)$.
\end{enumerate}
\end{theorem}

\begin{proof}
To prove statements 1 and 2, it is sufficient to prove them for the principal projective $P_m$ for all $m\leq d$.
Let $Q_m(-) := \C\{\Hom_{\FA}(-, [m])\}$, so that $P_m$ is a submodule of $Q_m$.  Then $Q_m(n) \cong (\C^m)^{\otimes n}$,
and Schur-Weyl duality tells us that the multiplicity of $V_\la$ in this representation is equal to the dimension of
the representation of $\operatorname{GL}(m; \C)$
indexed by $\la$.  In particular, it is zero unless $\la$ has at most $m$ parts, and the dimension of the representation
indexed by $\la(n)$ is bounded by a polynomial in $n$ of degree at most $m-1$.  Statements 1 and 2 follow for
$Q_m$, and therefore for $P_m$.

If $N'$ is finitely generated in degrees $\leq d$, then statement 3 holds for $N'$ by \cite[Corollary 8.1.4]{sam}.
If $N$ is a subquotient of $N'$, then it is still finitely generated in degrees $\leq r$ for some $r$, so statement 3
holds for $N$ with $d$ replaced by $r$.  But, since $N$ is a subquotient of $N'$, we have $\dim N(n) \leq \dim N'(n)$
for all $n$, which implies that $e_j=0$ for all $j\leq r$.
Statement 4 follows from statement 3 by finding a partial fractions decomposition of the ordinary generating function,
as observed in \cite[Remark 8.1.5]{sam}.  

To prove statement 5, it is again sufficient to consider $P_m$ for all $m\leq d$.
We have $$\dim P_m(n) = |\Hom_{\FS}([n],[m])| \leq |\Hom_{\FA}([n],[m])| = m^n \leq d^n.$$
Since $N$ is a subquotient of a finite direct sum of modules of this form, the dimension of $N(n)$
is bounded by a constant times $d^n$.
\end{proof}

We now record a pair of lemmas that say that certain natural constructions preserve smallness.

\begin{lemma}\label{general}
Fix a natural number $k$, a $k$-tuple of natural numbers $(d_1,\ldots,d_k)$, and a collection of $\FSop$-modules 
$N_1,\ldots,N_k$ such that $N_i$ is $d_i$-small.  Let $d = d_1+\cdots+d_k$.
Then the $\FSop$-module $N$ given by the formula 
$$N(E) := \bigoplus_{f:E\twoheadrightarrow[k]} N_1(f^{-1}(1))\otimes\cdots\otimes N_k(f^{-1}(k))$$
is $d$-small.
\end{lemma}

\begin{proof}
Since $d$-smallness is preserved by taking direct sums and passing to subquotients, we may assume that $N_i = P_{m_i}$
for some $m_i\leq d_i$.
Then \begin{eqnarray*}
N(E) &\cong& \bigoplus_{f:E\twoheadrightarrow[k]} P_{m_1}(f^{-1}(1))\otimes\cdots\otimes P_{m_k}(f^{-1}(k))\\
&\cong& \bigoplus_{f:E\twoheadrightarrow[k]} \C\left\{\Hom_{\operatorname{FS}}\!\big(f^{-1}(1),[m_1]\big)\right\}\otimes\cdots\otimes\C\left\{\Hom_{\operatorname{FS}}\!\big(f^{-1}(k),[m_k]\big)\right\}\\
&\cong& \bigoplus_{f:E\twoheadrightarrow[k]} \C\Big\{\Hom_{\operatorname{FS}}\!\big(f^{-1}(1),[m_1]\big)\times\cdots\times\Hom_{\operatorname{FS}}\!\big(f^{-1}(k),[m_k]\big)\Big\}\\
&\cong& \C\Big\{\Hom_{\operatorname{FS}}\!\big(E, [m_1]\sqcup\cdots\sqcup[m_k]\big)\Big\}\\
&\cong& \C\Big\{\Hom_{\operatorname{FS}}\!\big(E, [m_1+\cdots+m_k]\big)\Big\}\\
&\cong& P_{m_1+\cdots+m_k}(E),
\end{eqnarray*}
so $N$ is $d$-small.
\end{proof}

\begin{lemma}\label{shift}
Let $N$ be $d$-small and let $S$ be any set.  Let $N_S$ be the $\FS$-module defined by putting $N_S(E) := N(S\sqcup E)$
for all $E$, with maps defined in the obvious way.  Then $N_S$ is also $d$-small.
\end{lemma}

\begin{proof}
As in the proof of Lemma \ref{general}, we may reduce to the case where $N = P_m$ for $m \leq d$.
In this case, it is sufficient to show that every surjection $f:S\sqcup E\to [m]$ factors as $g\circ (\id_S \sqcup\, h)$,
where $g$ is a surjection from $S\sqcup [j]$ to $[m]$ for some $j\leq m$ and $h$ is a surjection from $[m]$ to $[j]$.
It is clear that we can do this by taking $j$ to be the cardinality of $f(E)$.
\end{proof}

\begin{remark}
The functor $N\mapsto N_S$ is called a {\bf shift functor}, and the analogous operation for FI-modules
has appeared in many contexts; see, for example, \cite[Section 2]{CEFN}.
\end{remark}

Finally, the following lemma will be needed in the proof of Theorem \ref{D-fg}.

\begin{lemma}\label{ri}
Suppose that $N\to N' \to N''$ is a complex of $d$-small $\FSop$-modules, and let $H$ denote its homology in the middle.
If $r_d(N)=0 = r_d(N'')$,
then $r_d(H) = r_d(N')$.
\end{lemma}

\begin{proof}
This follows from the fact that $\dim N'(n) - \dim N(n) - \dim N''(n) \leq \dim H(n) \leq \dim N(n)$
and the definition of $r_d$.
\end{proof}

\section{Configurations of points in the plane}\label{sec:config}
For any finite set $E$, let $\Conf(E)$ be the space of injective maps from $E$ to $\R^2$.
Arnol'd \cite{Arnold} proved that $$H^*(\Conf(E); \C) \cong \Lambda_\C\left[x_{jk}\mid j,k \in E\right] \Big{/} \Big\langle x_{jj},\;\; x_{jk} - x_{kj},\;\; x_{jk}x_{kl} + x_{kl}x_{lj} + x_{lj}x_{jk}\Big\rangle.$$
Let $H^i(E) := H^i(\Conf(E); \C)$ and $H_i(E) := H_i(\Conf(E); \C) \cong H^i(\Conf(E); \C)^*$.
Given a map $f:E\to F$, we have a map $H^*(\Conf(E); \C)\to H^*(\Conf(F); \C)$ taking $x_{jk}$ to $x_{f(j)f(k)}$.  This gives $H^i$ the structure
of an $\FA$-module and $H_i$ the structure of an $\FAop$-module.  Since $\FS$ is a subcategory of $\FA$, we may regard $H^i$ as an $\FS$-module
and $H_i$ as an $\FSop$-module.

\begin{proposition}\label{H-fg}
The $\FSop$-module $H_0$ is 1-small.  If $i\geq 1$, then $H_i$ is $2i$-small and $r_{2i}(H_i) = 0$.
\end{proposition}

\begin{proof}
We have $H_0\cong P_1$, so $H_0$ is 1-small.  For the remainder of the proof, fix a natural number $i>0$.
Since $H^*(E)$ is generated in degree 1, $H^i(E)$ is a quotient of $H^1(E)^{\otimes i}$. This means that $H_i(E)$ is a subspace of $H_1(E)^{\otimes i}$,
thus to prove $2i$-smallness it will suffice to show that $H_1^{\otimes i}$ is finitely generated in degrees $\leq 2i$.\footnote{In the published version
of this paper, we incorrectly asserted that it would suffice to show that $H_1$ is finitely generated in degree 2.  In fact, this would only prove that
 $H_1^{\otimes i}$ is finitely generated in degrees $\leq 2^i$.}

For any $j\neq k\in E$, let $e_{jk}\in H_1(E)$ be the linear functional that evaluates to 1 on $x_{jk}=x_{kj}$ and to 0 on all other $x_{mn}$.
Then $H_1(E)^{\otimes i}$ has a basis given by elements of the form $e_{j_1 k_1}\otimes\cdots \otimes e_{j_ik_i}$.  We will show that such a basis element
can be written as a linear combination of classes that are pulled back from sets of cardinality at most $2i$.  
We may assume that $|E|>2i$, otherwise there is nothing to show.

Let $F := \{j_1, k_1, \ldots, j_i, k_i\} \sqcup \{0\}$; this is a set of cardinality at most $2i+1$.  Define a surjection from $E$ to $F$ by sending each of the elements
of the set $\{j_1, k_1, \ldots, j_i, k_i\}$ to itself
and sending everything else to 0.  Then the class $e_{j_1 k_1}\otimes\cdots \otimes e_{j_ik_i}$ pulls back to itself.  If $|F|\leq 2i$, we are done.  
However, if $|F|=2i+1$, we still need to show that $e_{j_1 k_1}\otimes\cdots \otimes e_{j_ik_i}$ can be expressed as a sum of pullbacks of classes from smaller sets.  

Let $G$ be the set of cardinality $2i$ obtained from $F$ by identifying $k_i$ with $0$.  Then we have a canonical surjective map from $F$ to $G$, and the pullback
of $e_{j_1 k_1}\otimes\cdots \otimes e_{j_ik_i}$ along this map is equal to
$$e_{j_1 k_1}\otimes\cdots \otimes e_{j_ik_i}\; +\; e_{j_1 k_1}\otimes\cdots \otimes e_{j_i0}.$$
By symmetry, we can also obtain the following two classes as pullbacks along surjective maps to sets of cardinality $2i$:
\begin{align*}e_{j_1 k_1}\otimes\cdots \otimes e_{j_ik_i}\; &+\; e_{j_1 k_1}\otimes\cdots \otimes e_{k_i0}\\
e_{j_1 k_1}\otimes\cdots \otimes e_{j_i0}\; &+\; e_{j_1 k_1}\otimes\cdots \otimes e_{k_i0}.\end{align*}
By adding the first two classes, subtracting the third, and dividing by 2, one obtains the class $e_{j_1 k_1}\otimes\cdots \otimes e_{j_ik_i}$.
This completes the proof that $H_i$ is $2i$-small.

For the last statement, we begin by noting that $\dim H_1(n) = \binom{n}{2}$, therefore $$r_2(H_1^{\otimes i}) = \lim_{n\to \infty} (2i)^{-n}\binom{n}{2}^i = 0.$$
Since $H_i\subset H_1^{\otimes i}$, we have $r_{2i}(H_i)=0$, as well.
\end{proof}

\begin{remark}
The second statement of Proposition \ref{H-fg} also follows from the fact that $H^i$ is finitely generated as an FI-module
\cite[Theorem 6.2.1]{church-ellenberg-farb}.  (More generally, they prove this with $\R^2$ replaced
by any connected, oriented manifold of dimension greater than 1 with finite dimensional cohomology.)
This implies that the dimension of $H^i(n)$ grows as a polynomial in $n$
\cite[Theorem 1.5]{church-ellenberg-farb}, thus the same is true for the dimension of the $\FSop$-module $H_i(n)\cong H^i(n)^*$.
\end{remark}

For any $p\geq 0$, let
\begin{eqnarray*}
\Comp_{p,i}(E)\; &:=& \bigoplus_{f:E\twoheadrightarrow [p+1]} \Big(H_\bullet(f^{-1}(1))\otimes\cdots\otimes H_\bullet(f^{-1}(p+1))\Big)_{\!i}\\
&\cong& \bigoplus_{\substack{f:E\twoheadrightarrow [p+1]\\ i_1 + \cdots + i_{p+1} = i}} H_{i_1}(f^{-1}(1))\otimes\cdots\otimes H_{i_{p+1}}(f^{-1}(p+1)).
\end{eqnarray*}
It is clear that $\Comp_{p,i}$ comes endowed with a natural
$\FSop$-module structure.

\begin{proposition}\label{cp}
The $\FSop$-module $\Comp_{p,0}$ is $(p+1)$-small, and $\Comp_{p,i}$ is $(p+2i)$-small for all $i\geq 1$.
\end{proposition}

\begin{proof}
By Lemma \ref{general} and Proposition \ref{H-fg}
the summand of $\Comp_{p,i}$ corresponding to the tuple $(i_1,\ldots,i_{p+1})$ is $(d+2i)$-small, where $d$ is the number
of $k$ such that $i_k=0$.  When $i=0$, we have $d=p+1$.  When $i>0$, the maximum value of $d$ is $p$.
\end{proof}

\section{The main theorem}\label{sec:kl}
For any finite set $E$, let $\cI_{E} := \{(i,j)\mid i\neq j\in E\}$, and define $V_E\subset \C^{\cI_E}$
in a manner analogous to the definition of $V_n\subset \C^{\cI_n}$ in Section \ref{sec:kl}.  In particular, we have 
$\cI_{[n]} = \cI_n$ and $V_{[n]} = V_n$.  Define the reciprocal plane $X_E := X(V_E)$, and let
$D_i(E) := I\! H^{2i}\big(X_E; \C\big)$.  By Theorem \ref{ih}, $D_i(E)$ is the
$i^\text{th}$ $\Aut(E)$-equivariant Kazhdan-Lusztig coefficient
of the matroid $M_E$ associated with the complete graph on the vertex set $E$.  
In particular, if we take $E = [n]$, we have $D_i(n) = C^{S_n}_{M_n, i}$.

A surjective map of sets $E\to F$ is equivalent to the data of a partition of $E$ along with a bijection between $F$ and the set of parts
of the partition.  A partition of $E$ determines a flat of $M_E$, and the bijection between $F$ and the set of parts of the partition
determines an isomorphism from $X_F$ to $X\!\left((V_E)^F\right)$.  
Thus, Theorem \ref{canon}(1) gives us a map from $D_i(E)$ to $D_i(F)$,
and the first half of Theorem \ref{canon}(3) tells us that $D_i$ is an FS-module.

For any non-negative integers $p,q$, define 
$$A_i^{p,q}(E) := \Comp_{p,2i-p-q}(E) \otimes D_{i-q}^*(p+1).$$
Since $\Comp_{p,2i-p-q}$ is an $\FSop$-module with an action of the symmetric group $S_{p+1}$ (given by permuting the pieces of the composition)
and $D_{i-q}(p+1)^*$ is a fixed vector space equipped with an action of $S_{p+1}$, 
$A_i^{p,q}$ inherits the structure of an $\FSop$-module with an action of the symmetric group $S_{p+1}$.
Let $B_i^{p,q} := (A_i^{p,q})^{S_{p+1}}$ be the invariant submodule, and let $(B_i^{p,q})^*$ be the dual $\FS$-module.
By Corollary \ref{ss-braid}, we have a first quadrant cohomological spectral sequence with $E_1$ page $B_i^{p,q}(E)^*$
that converges to $D_i(E)$.  In other words, $D_i(E)$ admits a filtration whose associated graded is isomorphic to the $E_\infty$
page of this spectral sequence.
By the second half of Theorem \ref{canon}(3), each $(B_i^{p,q})^*$
admits the structure of an FS-module such that the FS-module maps commute with the differentials in the spectral sequence.
By Theorem \ref{canon}(4), the FS-module structure on $(B_i^{p,q})^*$ coming from Theorem \ref{canon}(3)
coincides with the FS-module structure that we defined explicitly.

\begin{theorem}\label{D-fg}
For all $i\geq 1$, the $\FSop$-module $D_i^*$ admits a filtration\footnote{In the published version of this paper, we claimed
$D_i^*$ is $2i$-small.  We do not know whether or not this is true.} whose associated graded is $2i$-small, and we have 
$$r_{2i}(D_i^*) = \frac{\dim D_{i-1}(2i)}{(2i)!}.$$
\end{theorem}

\begin{proof}
To prove that $D_i^*$ admits a filtration whose associated graded is $2i$-small, we will prove that the $E_\infty$ page of our (dualized) spectral sequence
is $2i$-small.
Since smallness is preserved under taking subquotients, it suffices to prove that the $E_1$ page is $2i$-small,
which means proving that $B_i^{p,q}$ is $2i$-small for all $p$ and $q$.
Since $B_i^{p,q}\subset A_i^{p,q}$, it suffices to prove it for $A_i^{p,q}$.  By Proposition \ref{cp} and the fact that smallness is preserved by taking a tensor
product with a fixed vector space, $A_i^{p,q}$ is $(p+1)$-small when $p+q=2i$ and $(p+2(2i-p-q))$-small otherwise.

Consider the case where $p+q=2i$.  
By definition of the equivariant Kazhdan-Lusztig polynomial, $D_i(E) = 0$ unless $2i<|E|-1$ or $|E|=1$ and $i=0$.
In particular, if $p=2i$ and $q=0$, then $D_{i-q}(p+1) = D_i(2i) = 0$, and therefore $A_i^{p,q} = 0$.  
Thus we may assume that $p<2i$.  Since $A_i^{p,q}$ is $(p+1)$-small
it is also $2i$-small.

Next, consider the case where $p+q<2i$, so $A_i^{p,q}$ is $(p+2(2i-p-q))$-small.
By the above vanishing property for $D_i(E)$, we have
$D_{i-q}(p+1) = 0$ unless $2(i-q) < p$ or $p=0$ and $q=i$.  
Thus we may conclude that
$A_i^{p,q}=0$
unless $$p + 2(2i-p-q)+p = 2(i-q) - p + 2i < 2i
\qquad\text{or}\qquad\text{$p=0$ and $q=i$}.$$  In particular, $A_i^{p,q}$ is $2i$-small, and therefore so is $D_i^*$.

This argument in fact proves that $A_i^{p,q}$ is $(2i-1)$-small unless $(p,q) = (0,i)$ or $(2i-1,1)$,
and the same is therefore true for $B_i^{p,q}$.  Furthermore, we have $B_i^{0,i} \cong H_i$, and 
Proposition \ref{H-fg} tells us that $r_{2i}(H_i)=0$.  Thus $r_{2i}(B_i^{p,q}) = 0$ unless $(p,q) = (2i-1,1)$,
and Lemma \ref{ri} therefore tells us that $r_{2i}(D_i^*) = r_{2i}(B_i^{2i-1,1})$.

We have $B_i^{2i-1,1} \cong (\Comp_{2i-1,0})^{S_{2i}}\otimes D^*_{i-1}(2i)$,
where $(\Comp_{2i-1,0})^{S_{2i}}$ is the $\FSop$-module that takes $E$ to a vector space
with basis given by partitions of $E$ into $2i$ nonempty pieces.  This means that $\dim (\Comp_{2i-1,0})^{S_{2i}}(n)$
is equal to the Stirling number of the second kind $S(n,2i)$, thus
$$r_{2i}(D_i^*) = r_{2i}(B_i^{2i-1,1}) = \lim_{n\to\infty}\frac{\dim B_i^{2i-1,1}(n)}{(2i)^n} = \lim_{n\to\infty}\frac{S(n,2i) \dim D_{i-1}(2i)}{(2i)^n}
= \frac{\dim D_{i-1}(2i)}{(2i)!},$$
and the theorem is proved.
\end{proof}

Let $H_i(u) := H_{D_i^*}(u)$ and $G_i(u) := G_{D_i^*}(u)$. 
Since representations of finite groups are self-dual, $H_i(u)$ and $G_i(u)$ may be regarded as generating functions
(ordinary and exponential) for the degree $i$ Kazhdan-Lusztig coefficients of braid matroids.
The following corollary follows from Theorems \ref{dsmall} and \ref{D-fg}, along with the fact that representations of finite groups are semisimple,
so passing to the associated graded of a filtered representation does not change the isomorphism type.

\begin{corollary}\label{cor}  Let $i$ be a positive integer.
\begin{enumerate}
\item If $\la\vdash n$ and $\Hom_{S_n}(V_\la, D_i(n)) \neq 0$, then $\ell(\la)\leq 2i$.  
\item For any partition $\la$ with $n \geq |\la| + \la_1 $, 
$\dim\Hom_{S_n}\!\big(V_{\la(n)}, D_i(n)\big)$ is bounded by a polynomial in $n$ of degree at most $2i-1$.
\item The ordinary generating function $H_i(u)$
is a rational function whose poles are contained in the set $\{1/j \mid 1\leq j \leq 2i\}$.
Furthermore, $H_i(u)$ has at worst a simple pole at $1/2i$.
\item There exists polynomials $p_0(u),\ldots,p_{2i}(u)$ such that the exponential generating function
$G_i(u)$ is equal to $$\sum_{j=0}^d p_j(u)e^{ju}.$$
Furthermore, $p_{2i}(u)$ is equal to the constant polynomial with value $r_{2i}(D_i^*) = \frac{\dim D_{i-1}(2i)}{(2i)!}$.
\end{enumerate}
\end{corollary}

\begin{remark}
Theorem \ref{D-fg} and Conjecture \ref{top} combine to say that
$$r_{2i}(D_i^*) = \frac{(2i-3)!!(2i-1)^{i-2}}{(2i)!} = \frac{(2i-1)^{i-3}}{2^i\, i!}.$$
In particular, if Conjecture \ref{top} is true (or more generally if $D_{i-1}(2i)\neq 0$), then $H_i(u)$ does have a pole at $1/2i$.
\end{remark}

\section{Examples}
We now example the cases when $i=1$ or 2 in greater detail. 

\begin{example}
We first consider the case when $i=1$.
In \cite[Proposition 4.4]{GPY}, we showed that $\Hom_{S_n}(V_\la, D_1(n)) = 0$ for all $\la$ with more than 2 rows,
and that $\dim\Hom_{S_n}\!\big(V_{[k](n)}, D_1(n)\big)$ is bounded by $n/2 + 1 - k$.
By \cite[Corollary 2.24]{EPW}, we have $\dim D_1(n) = 2^{n-1} - 1 - \binom{n}{2}$, which implies that 
$$H_1(u) = \frac{u^4}{(1-u)^3(1-2u)}$$
and
$$G_1(u) = \frac{1}{2} + \left(\frac{u^2}{2} - 1\right)e^u + \frac 1 2 e^{2u}.$$
In particular, 
$r_2(D_1^*) = 1/2 = \dim D_0(2)/2!$.\\
\end{example}

\begin{example}
We next consider the case when $i=2$.  By
\cite[Corollary 2.24]{EPW}, we have
$$\dim D_2(n) = s(n,n-2) - S(n,n-1)S(n-1,2) + S(n,3) + S(n,4),$$
where $s(n,k)$ and $S(n,k)$ are Stirling numbers of the first and second kind, respectively.
We have well-known generating function identities $$\sum_{n\geq 1} S(n,k) u^n = \frac{u^k}{\prod_{j=1}^k (1-ju)},$$
as well as \cite[A000914]{oeis}
$$\sum_{n\geq 1} s(n,n-2) u^n = \frac{2u^3+u^4}{(1-u)^5}.$$
Since $S(n,n-1)S(n-1,2) = \binom{n}{2}\left(2^{n-2}-1\right)$, it is not hard to show that
$$\sum_{n\geq 1}S(n,n-1)S(n-1,2) u^n = \frac{u^2}{(1-2u)^3} - \frac{u^2}{(1-u)^3}.$$
Putting it all together, we get
\begin{eqnarray*}H_{2}(u) &=& \frac{2u^3+u^4}{(1-u)^5} - \left(\frac{u^2}{(1-2u)^3} - \frac{u^2}{(1-u)^3}\right)\\
&& + \frac{u^3}{(1-u)(1-2u)(1-3u)} + \frac{u^4}{(1-u)(1-2u)(1-3u)(1-4u)}\\\\
&=& \frac{15u^6-50u^7 + 40u^8 + 4u^9}{(1-u)^5(1-2u)^3(1-4u)}.
\end{eqnarray*}
After performing a partial fractions decomposition we find that $r_{4}(D_2^*) = 1/24 = \dim D_1(4)/4!$.

We do not have a general formula for the dimension of $\Hom_{S_n}(V_\la, D_2(n))$, but we have computed $D_2(n)$ for all $n\leq 9$ 
\cite[Section 4.4]{GPY}, and it is indeed the case in these examples that the multiplicity of $V_\la$ in $D_2(n)$ is zero whenever $\la$ has more than 4 rows.
\end{example}

\section{The relative case}
Let $\Gamma$ be a finite graph with vertex set $V$. 
For any finite set $E$, let $\Gamma(E)$ be the graph with vertex set $V\sqcup E$ such that two elements of $V$ are adjacent
if and only if they were adjacent in $\Gamma$, and elements of $E$ are adjacent to everything.
We will define an FS-module structure on the $i^\text{th}$ $\Aut(E)$-equivariant
Kazhdan-Lusztig coefficient $D_i^\Gamma(E)$ of the matroid associated with the graph $\Gamma(E)$,
and prove that the dual $\FSop$-module admits a filtration whose associated graded is $2i$-small.
If $\Gamma$ is the empty graph, then $\Gamma(E)$ is just the complete graph on $E$, so we have $D_i^\Gamma = D_i$.

We begin by generalizing the material in Section \ref{sec:config}.  Let $\Gamma = (V,Q)$ be a finite graph with vertex set $V$ and edge set $Q$,
and let $\Conf(\Gamma)$ be the set of maps from $V$ to $\R^2$ that send adjacent vertices to distinct points.
We have the following description of the cohomology ring of $\Conf(\Gamma)$ \cite[Theorems 3.126 and 5.89]{OT}:
\begin{eqnarray*}
H^*(\Conf(\Gamma); \C) &\cong& \Lambda_\C[x_q]_{q\in Q} \Big{/} \left\langle\, \sum_{j=1}^k (-1)^j x_{q_1}\cdots \hat x_{q_j} \cdots x_{q_k} \;\;\Big{|}\;\; (q_1,\ldots,q_k)\;\text{a closed path}\right\rangle\\\\
&\cong& \text{the subring of all meromorphic differential forms on $\C^V$}\\ 
&& \text{generated by $\frac{dz_i-dz_j}{z_i-z_j}$ for all
$\{i,j\}\in Q$.}
\end{eqnarray*}
By definition, a map from $\Gamma = (V,Q)$ to $\Gamma' = (V',Q')$ is a map from $V$ to $V'$ that takes $Q$ to $Q'$.  Given a map $f:\Gamma\to\Gamma'$,
we obtain a map $H^*(\Conf(\Gamma); \C)\to H^*(\Conf(\Gamma'); \C)$ taking $x_q$ to $x_{f(q)}$.
In particular, we obtain an FA-module $H^i_\Gamma(E) := H^i(\Conf(\Gamma(E)); \C)$ and a dual $\FAop$-module 
$H_i^\Gamma(E) := H_i(\Conf(\Gamma(E)); \C)$.  
As in the case where $\Gamma$ is empty, we can regard $H^i_\Gamma$ as an FS-module and $H_i^\Gamma$ as an $\FSop$-module.
The proof of the following proposition is identical to the proof of Proposition \ref{H-fg}.

\begin{proposition}\label{H-fg-Gamma}
The $\FSop$-module $H_0^\Gamma$ is 1-small.  If $i\geq 1$, then $H_i^\Gamma$ is $2i$-small and $r_{2i}(H_i^\Gamma) = 0$.
\end{proposition}

Given a graph $\Gamma$ with vertex set $V$ and a subset $S\subset V$, let $\Gamma_S$ be the induced subgraph with vertex set $S$.
Given a surjective map $f:V\to V'$,
let $\Gamma^f$ be the graph with vertex set $V'$ whose edges are the images of edges of $\Gamma$ (ignoring loops and multiple edges).
Fix a graph $\Delta$ with vertex set $[p+1]$, and define
$$\CG(E) \;\; := \bigoplus_{\substack{f:V\sqcup E\twoheadrightarrow [p+1]\\ \Gamma(E)^f = \Delta\\ \text{$\Gamma(E)_{f^{-1}(j)}$ connected $\forall j$}
}}
H_i\Big(\Conf\!\left(\Gamma(E)_{f^{-1}(1)}\right)\times\cdots\times\Conf\!\left(\Gamma(E)_{f^{-1}(p+1)}\right); \C\Big).$$

Given surjective maps $g:E\to F$ and $f:V\sqcup F\to [p+1]$ such that $\Gamma(E)_{f^{-1}(j)}$ is connected for all $j$, we can compose $f$ with $g$ to
obtain a surjective map $g^*f:V\sqcup E\to [p+1]$
with the property that $\Gamma(E)_{(g^*f)^{-1}(i)}$ is connected for all $j$ 
and $\Gamma(E)^{g^*f} = \Gamma(F)^f$.  This observation allows us to define an $\FSop$-module
structure on $\CG$.  Taking $\Gamma$ to be the empty graph and $\Delta$ the complete graph,
we have $\Comp_{p,i}^{\Gamma,\Delta} = \Comp_{p,i}$.
The following proposition generalizes Proposition \ref{cp}.

\begin{proposition}\label{cp-Gamma}
The $\FSop$-module $\Comp^{\Gamma,\Delta}_{p,0}$ is $(p+1)$-small, and $\CG$ is $(p+2i)$-small for all $i\geq 1$.
\end{proposition}

\begin{proof}
Let $\Comp_{p,i}^{\Gamma} := \bigoplus_{\Delta} \CG$.  We will prove that $\Comp_{p,i}^{\Gamma}$ is $(p+1)$-small
when $i=0$ and $(p+2i)$-small when $i\geq 1$,
and therefore so is each of its summands.  The above description of the cohomology ring of $\Conf(\Gamma)$
in terms of meromorphic differential forms makes it clear that $H^*(\Conf(\Gamma); \C)$ is a subring
of $H^*(\Conf(V); \C)$, and therefore that the $f$-summand of $\CG(E)$ is a quotient of 
the $f$-summand of $\Comp_{p,i}(V\sqcup E)$.
The proposition then follows from Proposition \ref{cp} and Lemma \ref{shift}.
\end{proof}

We next generalize the material in Section \ref{sec:kl}.
For any finite set $E$ and any non-negative integers $p,q$, define 
$$A_{\Gamma,i}^{p,q}(E)\; := 
\;\; \bigoplus_\Delta\; \Comp^{\Gamma,\Delta}_{p,2i-p-q}(E) \;\otimes\; D_{i-q}^{\Delta}\big(\emptyset\big)^*.$$
As in the case where $\Gamma$ is the empty graph, $A_{\Gamma,i}^{p,q}$ is an $\FSop$-module with an action of $S_{p+1}$,
and we define  the invariant $\FSop$-module $B_{\Gamma,i}^{p,q} := (A_i^{p,q})^{S_{p+1}}$ 
along with its dual FS-module $(B_{\Gamma,i}^{p,q})^*$.
There is again a first quadrant cohomological spectral sequence with $E_1$ page $B_{\Gamma,i}^{p,q}(E)^*$
that converges to $D^\Gamma_i(E)$, inducing an FS-module structure on $D^\Gamma_i$.  

\begin{theorem}
Let $\Gamma$ be a graph with vertex set $V$.
For all $i\geq 1$, the $\FSop$-module $(D^\Gamma_i)^*$ admits a filtration whose associated graded is $2i$-small, and we have
$$r_{2i}\big((D^\Gamma_i)^*\big) = \frac{(2i)^{|V|} \dim D_{i-1}(2i)}{(2i)!} = (2i)^{|V|} r_{2i}(D_i^*).$$
\end{theorem}

\begin{proof}
The same argument that we used in the proof of Theorem \ref{D-fg} shows that $(D^\Gamma_i)^*$ admits a filtration whose associated graded is $2i$-small
and $r_{2i}\big((D^\Gamma_i)^*\big) = r_{2i}(B_{\Gamma,i}^{2i-1,1})$.
Explicitly, we have
$$B_{\Gamma,i}^{2i-1,1}(E)\; =\left(\bigoplus_{f:V\sqcup E\twoheadrightarrow [2i]} D_{i-1}^{\Gamma(E)^f}\!(\emptyset)^*\right)^{S_{2i}}.$$
When $E$ is large, $\Gamma(E)_{f^{-1}(j)}$ is connected for all $j$ and $\Gamma(E)^f$ is equal to $K_{2i}$ for almost all maps 
$f:V\sqcup E\twoheadrightarrow [2i]$,
and the number of such maps is asymptotic to $(2i)^{|V|+n}$.
We therefore have
$$r_{2i}(B_{\Gamma,i}^{2i-1,1}) = \lim_{n\to\infty} \frac{(2i)^{|V|+n} \dim D_{i-1}(2i)}{(2i)^n (2i)!} = \frac{(2i)^{|V|} \dim D_{i-1}(2i)}{(2i)!},$$
and the theorem is proved.
\end{proof}

\bibliography{./symplectic}
\bibliographystyle{amsalpha}

\end{document}